\documentclass[11pt,a4paper]{article}
\pdfoutput=1
\usepackage{amssymb,amsmath,amsthm}
\usepackage[pdfstartview=FitH,colorlinks=true,linkcolor=black,anchorcolor=black,citecolor=black,urlcolor=black]{hyperref}
\usepackage[top=2.5cm,right=3cm,left=3cm,bottom=2.5cm]{geometry}
\usepackage[english]{babel}

\numberwithin{equation}{section}
\allowdisplaybreaks[4]

\usepackage[titles]{tocloft}

\cftsetpnumwidth{9pt}
\addto\captionsenglish{%
}
\newtheorem{prop}{Proposition}[section]

\newtheorem{lemma}[prop]{Lemma}

\newtheorem{df}[prop]{Definition}
\newtheorem{rem}[prop]{Remark}

\newcommand{\B}{\mathcal{B}}
\newcommand{\HH}{\mathcal{H}}
\newcommand{\E}{\mathcal{E}}
\newcommand{\U}{\mathcal{U}}
\newcommand{\mc}{\mathcal}
\newcommand{\mf}{\mathfrak}
\newcommand{\Q}{\mathbb{Q}}
\newcommand{\R}{\mathbb{R}}
\newcommand{\C}{\mathbb{C}}
\newcommand{\N}{\mathbb{N}}
\newcommand{\Z}{\mathbb{Z}}
\newcommand{\T}{\mathbb{T}}

\newcommand{\de}{\mathrm{d}}
\renewcommand{\u}{\underline}
\newcommand{\inner}[1]{\left<\smash[t]{#1}\right>}
\newcommand{\az}{\triangleright}
\newcommand{\za}{\triangleleft}
\newcommand{\bza}{\,\textrm{\footnotesize\raisebox{0.5pt}{$\blacktriangleleft$}}\,}

\renewcommand{\Im}{\mathrm{Im}}

\begin{document}

\title{\vspace*{-1cm}Modules over the Noncommutative Torus and Elliptic Curves \\[10pt] }

\author{Francesco D'Andrea$^{1,2}$, Gaetano Fiore$\hspace{1pt}^{1,2}$ and Davide Franco$\hspace{1pt}^1$ \\[12pt]
{\footnotesize $\hspace{-6pt}^1\,$Dipartimento di Matematica e Applicazioni, Universit\`a di Napoli {\sl Federico II}, Via Cintia, 80126 Napoli.} \\[3pt]
{\footnotesize $^2\,$I.N.F.N., Sezione di Napoli, Complesso MSA, Via Cintia, 80126 Napoli}}

\date{}

\maketitle

{\renewcommand{\thefootnote}{}
\footnotetext{%
\hspace*{-6pt}\textit{MSC-class 2010:} Primary: 58B34; Secondary: 46L87, 53D55. \\
\hspace*{10pt}\textit{Keywords:} Noncommutative torus, imprimitivity bimodules, elliptic curves, Moyal deformation.}}

\begin{abstract}\noindent
Using the Weil-Brezin-Zak transform of solid state physics, we describe line bundles over elliptic curves in terms of Weyl operators.
We then discuss the connection with finitely generated projective modules over the algebra $A_\theta$ of the noncommutative torus.
We show that such $A_\theta$-modules have a natural interpretation as Moyal deformations of vector bundles over an elliptic curve $E_\tau$, under the condition that the deformation parameter $\theta$ and the modular parameter $\tau$ satisfy a non-trivial relation.
\end{abstract}

\bigskip

\section{Introduction}\label{sec:1}
In the study of the ergodic action of a topological group $G$ on a compact space $M$, the orbit space can be
efficiently described by using the crossed product $C^*$-algebra $C(M)\rtimes G$. The most celebrated
example is the irrational rotation algebra $A_\theta$, arising from the action of $\Z$ on $\R/\Z$
generated by an irrational translation $x\mapsto x+\theta$.
This is a strict deformation quantization of $C(\T^2)$~\cite{Rie90},
hence the name ``noncommutative torus'', and is one of several examples of quantization of
manifolds carrying an action of $\T^2$ or more generally $\R^2$~\cite{Rie93}.

Finitely generated and projective $A_\theta$-modules have been classified by Connes and Rieffel in~\cite{Con80,Rie83}.
From a geometric point of view, it is natural to wonder whether they can be obtained as deformations of vector bundles
over the torus $\T^2$.
Line bundles over tori $\T^n$ have a nice physical interpretation as their $L^2$-sections form
Hilbert spaces describing (electrically) charged quantum particles on $\T^n$ in the presence of a non-zero
magnetic field, see \cite{Fio13} for a physically-minded presentation.

An obstruction in applying the standard quantization technique for the action of $\T^2$ resp.~$\R^2$~\cite{Rie93}
is given by the fact that the only finitely generated projective $C(\T^2)$-modules carrying an action of $\R^2$
are the free modules. More precisely, sections of a non-trivial line bundle $L\to\T^2$, when realized as quasi-periodic
functions in $C^\infty(\R^2)$, are not stable under the generators $\partial_x,\partial_y$ of the action of $\R^2$ by translation.

One possible solution is to replace the ordinary derivatives with ``covariant'' derivatives, cf.~equation \eqref{eq:nabla}. These
generate a projective representation of $\R^2$ on sections of line bundles, that is an proper representation of the
$3$-dimensional Heisenberg group $H_3$. One might then use a suitable Drinfeld twist \cite{Dri89} based on $\U(\mf{h}_3)[[\hbar]]$, with $\mf{h}_3$ the
Lie algebra of $H_3$, to deform vector bundles over $\T^2$ into modules for the noncommutative torus. 
A detailed study of Heisenberg twists and their use to study the dynamics of charged particles on the
noncommutative torus, along the lines of \cite{Fio11,Fio13}, is postponed to future works.

In this paper, we follow a different approach and realize finitely generated projective $A_\theta$-modules as deformations for an action of $\R^2$
of line bundles over an elliptic curve $E_\tau$, provided the modular parameter $\tau$ is chosen appropriately.
For a line bundle with degree $p$, in order for this construction to work the modular parameter and the deformation parameter $\tau$ must satisfy the constraint $\tau-\tfrac{p\theta}{2}i\in\Z+i\Z$.

Let us stress that here we are not interested in complex structures on the noncommutative torus itself, but only on how to realize its modules as non-formal deformations.
The reader interested in complex structures on noncommutative tori, in connection with elliptic curves, can see e.g.~the works of Polishchuk, Schwarz, Vlasenko and Plazas \cite{PS02,Pol02,Vla06,Pla06}.
For the relation between imprimitivity bimodules and \emph{noncommutative} elliptic curves one can see \cite{MvS08}.
Finally, for deformations of modules carrying an action of $\R^2$, for a general algebra acted by $\R^2$ via automorphisms,
one can see \cite{LW11}.

The paper is organized as follows.
In \S\ref{sec:2}, we recall definition and properties of the noncommutative torus and its modules.
We will be as self-contained as possible, with the purpose of both fixing notations and facilitating the reader.
In \S\ref{sec:3}, using the Weil-Brezin-Zak transform \eqref{eq:series}, we give a description of the module of sections
of a line bundles over an elliptic curve in terms of Weyl operators, much in the spirit of \S\ref{sec:4.1}, cf.~Prop.~\ref{prop:hat}.
In \S\ref{sec:5}, we show how to obtain modules isomorphic to the ones in \S\ref{sec:4.1} from
a deformation construction that makes use of Moyal star product \eqref{eq:Moyalprod}; in particular,
every module of \S\ref{sec:4.1} can be obtained in this way.
In Prop.~\ref{prop:4.4} we explain how to obtain a Hermitian structure similar to the Hermitian structure \eqref{eq:herstr} used in~\cite{Con80,Rie83}
from the canonical Hermitian structure of $C^\infty(\R^2)$.


\section{The noncommutative torus}\label{sec:2}
In this section, we collect some basic facts about the noncommutative torus, the description of
its algebra using Weyl operators, and recall the construction of finitely generated projective
modules, that in our framework replace vector bundles.

\subsection[The abstract C*-algebra and Weyl operators]{The abstract $C^*$-algebra and Weyl operators}
Let $0\leq\theta<1$. We denote by $A_\theta$ the universal unital $C^*$-algebra generated
by two unitary operators $U$ and $V$ with commutation relation
$$
UV=e^{2\pi i\theta}VU \;.
$$
If $\theta=0$, $A_0\simeq C(\T^2)$ is isomorphic to the $C^*$-algebra of continuous functions on a $2$-torus, with standard operations and sup norm.
Rather than $A_\theta$, we are interested in the subset
$A^\infty_\theta\subset A_\theta$, whose elements are series
$$
\sum\nolimits_{m,n\in\Z}a_{m,n}U^mV^n \;,
$$
where $\{a_{m,n}\}_{m,n\in\Z}$ is a sequence for which the norm
$$
p_k(a)^2=\sup_{m,n\in\Z}(1+m^2+n^2)^k|a_{m,n}|^2
$$
is finite for all $k\in\N$. 
We refer to such a sequence as ``rapid decreasing'',
and to the map $(m,n)\mapsto a_{m,n}$ as Schwartz function on $\Z^2$.
The collection of all sequences satisfying the condition above will be
denoted by $\mc{S}(\Z^2)$.
The set $A^\infty_\theta$ with the family of norms $p_k$ is a Fr\'echet pre $C^*$-algebra
(cf.~e.g.~\cite{Con94,Var06} and references therein).

A concrete description of this algebra can be obtained by introducing Weyl operators.
Let $a,b\in\R$. The Weyl operator $W(a,b)$ is the unitary operators on $L^2(\R)$ defined by
\begin{equation}\label{eq:WeylOP}
\bigl\{W(a,b)\psi\bigr\}(t)=e^{-\pi iab}e^{2\pi ibt}\psi(t-a) \;,\qquad\psi\in L^2(\R).
\end{equation}
Since
\begin{equation}\label{eq:Weylprod}
W(a,b)W(c,d)=e^{-\pi i(ad-bc)}W(a+c,b+d) \;,
\end{equation}
the linear span of all Weyl operators is a unital $*$-algebra.
A faithful unital $*$-representation $\pi$ of $A_\theta^\infty$
is given on generators by $\pi(U):=W(1,0)$ and $\pi(V):=W(0,-\theta)$.

The difference between the rational 
and irrational 
case can be understood by considering the associated von Neumann algebra
$
\mathcal{N}_\theta=\big\{W(m,n\theta):m,n\in\Z\big\}'' ,
$
closure of $\pi(A_\theta^\infty)$. If $\theta=p/q$ with $p,q$
coprime, it can be shown that the center is the
subalgebra generated by the operators $W(q,0)$ and $W(0,p)$,
isomorphic to $L^\infty(\T^2)$.
On the other hand, if $\theta$ is irrational, the center is trivial and $\mathcal{N}_\theta$ is a factor
(in fact, a type $\text{II}_1$ factor).

\smallskip

Let us conclude this section by recalling the deformation point of view.
Let $x,y\in\R^2$ and $u,v$ be the functions
\begin{equation}\label{eq:uandv}
u(x,y):=e^{2\pi ix} \;, \qquad
v(x,y):=e^{2\pi iy} \;.
\end{equation}
They are generators of the unital $C^*$-algebra $C(\T^2)$ (any periodic continuous function
is a uniform limit of trigonometric polynomials). By standard Fourier analysis, we know that
any $f\in C^\infty(\T^2)$ can be written as
\begin{equation}\label{eq:Fourier}
f=\sum\nolimits_{m,n\in\Z}a_{m,n}u^mv^n \;,
\end{equation}
and this series converges (uniformly) to a smooth function
$f$ if and only if $\{a_{m,n}\}\in\mc{S}(\Z^2)$.
There is an associative product $\ast_\theta$ on $C^\infty(\T^2)$ defined on monomials by
\begin{equation}\label{eq:asttheta}
(u^jv^k)\ast_\theta(u^mv^n)=\sigma((j,k),(m,n))u^{j+m}v^{k+n}
\;,\qquad
\text{for all}\;j,k,m,n\in\Z \;,
\end{equation}
where $\sigma:\Z^2\times\Z^2\to\C^*$ is a $2$-cocycle in the group cohomology complex of $\Z^2$, given by:
$$
\sigma((j,k),(m,n)):=e^{i\pi\theta(jn-km)} \;.
$$
One can then interpret the algebra $(C^\infty(\T^2),\ast_\theta)$ as a cocycle quantization
of $\mc{S}(\Z^2)$ with convolution product. The $C^*$-completion is the twisted group $C^*$-algebra $C^*(\Z^2,\sigma)$.
Note that since $(f\ast_\theta g)^*=g^*\ast_\theta f^*$, we get a $*$-algebra with undeformed involution.

A unital $*$-algebra isomorphism $T_\theta:(C^\infty(\T^2),\ast_\theta)\to A_\theta^\infty$ is given
on $f$ as in \eqref{eq:Fourier} by
\begin{equation}\label{eq:Ttheta}
T_\theta(f)
:=\sum\nolimits_{m,n\in\Z}a_{m,n}e^{-\pi imn\theta}U^mV^n \;,
\end{equation}
where the phase factors are chosen to have $T_\theta(f)^*=T_\theta(f^*)$ for all $f$.

The product \eqref{eq:asttheta} can be extended to a larger class of functions as explained in \cite{Rie93}.
A (well-defined) associative product on the Schwartz space $\mc{S}(\R^2)$ is given by:
\begin{equation}\label{eq:Moyalprod}
(f\ast_\theta g)(z)=\frac{4}{\theta^2}\int_{\C\times\C}f(z+\xi)g(z+\eta)e^{\frac{4\pi i}{\theta}\Im(\xi\hspace{1pt}\overline{\eta})}
\de\xi\hspace{1pt}\de\eta \;,
\end{equation}
where in complex coordinates $z=x+iy$ and $\de z=\de x\,\de y$. We refer to \eqref{eq:Moyalprod} as the ``Moyal product''.
This product is extended by duality to tempered distributions, and this allows to define an associative product, by restriction,
on several interesting function spaces. A relevant example is the space $\B(\R^2)\subset C^\infty(\R^2)$ of smooth functions
that are bounded together with all their derivatives. It is a theorem \cite[Prop.~2.23]{GGISV04} that $\B(\R^2)$ with the product above
(and suitable seminorms) is a unital Fr\'echet pre $C^*$-algebra. If $f,g$ are periodic, we recover the product \eqref{eq:asttheta}, cf.~\cite[Eq.~(2.12)]{GGISV04},
thus justifying using the same symbol for \eqref{eq:asttheta} and \eqref{eq:Moyalprod}.

\subsection{Heisenberg modules}\label{sec:4.1}
Finitely generated projective $A_\theta^\infty$-modules were constructed in~\cite{Con80},
and then completely classified in~\cite{Rie83}. By Serre-Swan theorem they provide a
noncommutative analog of complex smooth vector bundles.
The stable range theorem tells us that, on a base space $X$ of
real dimension $2$, any two complex vector bundles that are stably equivalent are also isomorphic,
i.e.~the map $\mathrm{Vect}(X)\to K^0(X)$ is injective. As a consequence, since
$K^0(\T^2)\simeq\Z\oplus\Z$, finitely generated projective $C^\infty(\T^2)$-modules
(complex vector bundles on $\T^2$) are classified by two integers, the rank $r$ and the first
Chern number $d$, or ``degree''.
The tensor product of vector bundles gives $K^0(\T^2)$ the structure of a unital ring,
isomorphic to $\Z[t]/(t^2)$ via the map $(r,d)\mapsto r+td$~\cite{Ati57}.

Something similar holds for $A_\theta^\infty$, $\theta\neq 0$. It is proved in~\cite{Rie83} that
$K_0(A_\theta^\infty)$ has the cancellation property, meaning that the map from
equivalence classes of finitely generated projective modules to the $K_0$-group
is injective, and since $K_0(A_\theta^\infty)\simeq\Z\oplus\Z$, these are also
classified by two integers $p,s$. The explicit definition of the corresponding modules,
here denoted $\E_{p,s}$, is recalled below,\footnote{Here we use the notations of
\cite[\S7.3]{Var06}, except for a replacement $\theta\to -\theta$.} and the aim of this paper
will be to describe them as deformations of vector bundles in a suitable sense.

\begin{df}[\cite{Con80,Rie83}]\label{def:heismod}
Let $p,s\in\Z$ and $p\geq 1$.
As a vector space, $\E_{p,s}:=\mc{S}(\R)\otimes\C^p$ with $\mc{S}(\R)$ the set of
Schwartz functions on $\R$. The right $A^\infty_\theta$-module structure is given by
$$
\psi\za U=\big\{W(\tfrac{s}{p}+\theta,0)\otimes (S^*)^s\big\}\psi \;,\qquad
\psi\za V=\big\{W(0,1)\otimes C\big\}\psi \;,
$$
where $C,S\in M_p(\C)$ are the \emph{clock} and \emph{shift} operators:
\begin{align}\label{eq:CS}
C&=
\text{\small $\begin{pmatrix}
1 & 0 & 0 & \ldots & 0 \\
0 &  e^{2\pi i/p} & 0 & \ldots & 0 \\
0 & 0 &  e^{4\pi i/p} & \ldots & 0 \\
\ldots & \ldots & \ldots & \ldots & \ldots \\
0 & 0 & 0 & \ldots &  e^{2(p-1)\pi i/p}
\end{pmatrix}$} \;,&
S&=
\text{\small $\begin{pmatrix}
0 & 0 & 0 & \ldots & 0 & 1 \\
1 & 0 & 0 & \ldots & 0 & 0 \\
0 & 1 & 0 & \ldots & 0 & 0 \\
0 & 0 & 1 & \ldots & 0 & 0 \\
\ldots & \ldots & \ldots & \ldots & \ldots & \ldots \\
0 & 0 & 0 & \ldots & 1 & 0
\end{pmatrix}$} \;,
\end{align}
and we adopt the convention that $C=S=1$ for $p=1$.
\end{df}

\noindent
If $\theta$ is irrational, any finitely generated projective right $A_\theta$-module is
isomorphic either to a free module $(A_\theta^\infty)^p$ or to
a module $\E_{p,s}$, with $p$ and $s$ coprime ($p>0$ and $s\neq 0$) or $p=1$ and $s=0$~\cite{Rie83,CR87}.
There is also a pre-Hilbert module structure on $\E_{p,s}$ recalled below.

\begin{df}[\cite{Con80,Rie83}]\label{prop:4.4}
An $A_\theta^\infty$-valued Hermitian structure on $\E_{p,s}$ is given by
\begin{equation}\label{eq:hstr}
\inner{\psi,\varphi}=\sum_{m,n\in\Z}
U^m V^n\int_{-\infty}^\infty
(\psi\za U^mV^n|\varphi)_t\,
\de t \;.
\end{equation}
where $(\,|\,):\E_{p,s}\times\E_{p,s}\to \mathcal{S}(\R)$ is
the canonical Hermitian structure of $\E_{p,s}=\mathcal{S}(\R)\otimes\C^p$
as a right pre-Hilbert $\mathcal{S}(\R)$-module,
given by
\begin{equation}\label{eq:canher}
(\psi|\varphi)_t:=\sum\nolimits_{r=1}^p\overline{\psi_r(t)}\varphi_r(t) \;,
\end{equation}
for all $\psi=(\psi_1,\ldots,\psi_p)$ and $\varphi=(\varphi_1,\ldots,\varphi_p)\in\E_{p,s}$.
\end{df}

To the best of our knowledge, for $\theta\neq 0$ there is no analog of the ring of vector bundles.
The analog of the tensor product of vector bundles would be the tensor product of bimodules over the algebra
itself, but while in the commutative case every module is a bimodule, this is not true for a general noncommutative algebra.

It is known, for example, that $A_\theta$ has non-trivial Morita auto-equivalence bimodules if and only if
$\theta$ is a real quadratic irrationality, i.e.~$\Q(\theta)$ is a real quadratic field~\cite{Man04}.
The strong analogy with the theory of complex multiplication of elliptic curves suggests that
noncommutative tori may play a role in number theory similar to the one played by elliptic curves.
This idea is the starting point of Manin real multiplication program~\cite{Man04}.

Constructing an analog of the ring of vector bundles is the typical problem that could be addressed, at least
formally, with the use of a Drindeld twist.


\section{Vector bundles over elliptic curves}\label{sec:3}
In this section we recall some basic facts about elliptic curves, in particular
about non-holomorphic factors of automorphy, and establish a
correspondence between modules of sections of line bundles and Heisenberg modules.

\subsection{Smooth and holomorphic line bundles over elliptic curves}

Let us identify $z=x+iy\in\C$ with the point $(x,y)\in\R^2$. The set
of complex-valued smooth functions on $\R^2\simeq\C$ will be denoted by $C^\infty(\C)$,
rather than $C^\infty(\R^2)$, when complex coordinates are used.

Fix $\tau\in\C$ with $\mathrm{Im}(\tau)>0$, let
$\Lambda:=\Z+\tau\Z$ and $E_\tau=\C/\Lambda$ the
corresponding elliptic curve with modular parameter
$\tau$. We use the symbol $\T^2$ when $\tau=i$.
Equivalently (Jacobi uniformization)
$E_\tau\simeq\C^*/q^{2\Z}$ where
$$
q=e^{\pi i\tau} \;,
$$
the biholomorphism being given by the exponential map
$z\mapsto e^{2\pi iz}$.

The algebra $C^\infty(E_\tau)$ can be identified with the subalgebra
of $C^\infty(\C)$ made of $\Lambda$-invariant functions:
$$
C^\infty(E_\tau)=\big\{f\in C^\infty(\C)\,:\,
f(z+m+n\tau)=f(z)\,\forall\,m,n\in\Z\big\} \;,
$$
The space $C^\infty(\C)$ has a natural structure of $C^\infty(E_\tau)$-module.

Let $\alpha:\Lambda\times\C\to\C^*$ be a smooth function and $\,\pi_\alpha:\Lambda\to\mathrm{End}\,C^\infty(\C)\,$ given by
\begin{equation}\label{eq:pi}
\pi_\alpha(\lambda)f(z):=\alpha(\lambda,z)^{-1}f(z+\lambda) \;,\qquad\forall\;\lambda\in\Lambda,z\in\C\;.
\end{equation}
Then $\pi_\alpha$ is a representation of the abelian group $\Lambda$ if and only if
\begin{equation}\label{eq:alphaeq}
\alpha(\lambda+\lambda',z)=\alpha(\lambda,z+\lambda')\alpha(\lambda',z) \;,
\qquad\forall\;z\in\C, \;\lambda,\lambda'\in\Lambda.
\end{equation}
An $\alpha$ satisfying \eqref{eq:alphaeq} is called a \emph{factor of automorphy} for $E_\tau$.
In the notations of~\cite{BL04},
this is an element of $Z^1(\Lambda,H^0(\mathcal{O}^*_V))$, where $V=\C$.
Consider the corresponding set of smooth quasi-periodic functions:
\begin{equation}\label{eq:smsec}
\Gamma_\alpha:=\big\{f\in C^\infty(\C):
\pi_\alpha(\lambda)f=f
\;\forall\;\lambda\in\Lambda 
\big\} \;.
\end{equation}
Then $\Gamma_\alpha$ is stable under multiplication by $C^\infty(E_\tau)$,
and hence a $C^\infty(E_\tau)$-submodule of $C^\infty(\C)$. It is projective
and finitely generated, since its elements can be thought as smooth sections
of a line bundle on $E_\tau$ (Appell-Humbert theorem).

Holomorphic elements of $\Gamma_\alpha$ are the so-called \emph{theta functions} for the
factor $\alpha$: they form a finite-dimensional vector space (while $\Gamma_\alpha$ is
finitely generated, but not finite-dimensional). Elements of $\Gamma_\alpha$ are called
\emph{differentiable theta functions}~\cite[pag.~57]{BL04}.

The basic example is the Jacobi theta function~\cite{Mum83}:
$$
\vartheta(z;q)=\sum_{n\in\Z}q^{n^2}e^{2\pi i nz} \;,\qquad\quad q=e^{\pi i\tau}\;.
$$
It satisfies
$$
\vartheta(z+m+n\tau;q)=q^{-n^2}e^{-2\pi i nz}\vartheta(z;q) \;.
$$
Hence the factor of automorphy is
\begin{equation}\label{eq:alpha}
\alpha(\lambda,z)=q^{-n^2}e^{-2\pi inz} \;,\qquad\quad z=x+iy,\;\lambda=m+n\tau .
\end{equation}
This corresponds to a holomorphic line bundle with degree $1$. If we forget
about the holomorphic structure, this is the unique smooth line bundle with degree
$1$ modulo isomorphisms.
Let $\tau=\omega_x+i\omega_y$, with $\omega_x,\omega_y\in\R$ and $\omega_y>0$.
The isomorphism of $C^\infty(E_\tau)$-modules $f\mapsto e^{-\pi y^2/\omega_y}f$ maps $\Gamma_\alpha$
into $\Gamma_\beta$, with
\begin{equation}\label{eq:beta}
\beta(\lambda,z)=e^{-\pi i\omega_xn^2}e^{-2\pi inx} \;,\qquad\quad z=x+iy,\;\lambda=m+n\tau,
\end{equation}
a unitary factor of automorphy.
Note that a $C^\infty(E_\tau)$-module isomorphism corresponds (is dual to) an isomorphism of smooth
vector bundles, but not necessarily of holomorphic vector bundles. In particular, 
$\Gamma_\beta$ are not (smooth) sections of an holomorphic vector bundle, and none of its
elements is holomorphic (cf.~\S\ref{sec:4.3}).

The general smooth line bundle with degree $p$ can be obtained from the factor of automorphy
$\beta^p$.


\subsection{Sections of line bundles and Heisenberg modules}

Let us denote by $\mc{F}_{\tau,p}$ the $C^\infty(E_\tau)$-module \eqref{eq:smsec}
associated to the factor of automorphy $\beta^p$, with $p\in\Z$ and $\beta$ as in
\eqref{eq:beta}. So, $f\in C^\infty(\C)$ belongs to $\mc{F}_{\tau,p}$ if{}f:
\begin{equation}\label{eq:expl}
f(z+1)=f(z)\quad\text{and}\quad
f(z+\tau)=e^{-\pi ip(\omega_x+2x)}f(z) \quad\forall\;z=x+iy\in\C.
\hspace{-5pt}
\end{equation}
Every $f\in\mc{F}_{\tau,p}$ is a bounded function, since from \eqref{eq:expl} it follows that
$|f|$ is a continuous periodic function.
Two maps $\nabla_1,\nabla_2:\mc{F}_{\tau,p}\to\mc{F}_{\tau,p}$ are given by:
\begin{equation}\label{eq:nabla}
\nabla_1f(z)=(\partial_x+\tfrac{2\pi ip}{\omega_y}y)f(z) \;,\qquad
\nabla_2f(z)=\partial_yf(z) \;.
\end{equation}
One can explicitly check that for any $f$ satisfying \eqref{eq:expl}, $\nabla_1f$ and $\nabla_2f$ satisfy \eqref{eq:expl} too.

In the next proposition, we study a transform very close to the
\emph{Weil-Brezin-Zak transform} used in solid state physics \cite[\S1.10]{Fol89}
(see also the end of section 2 in \cite{Fio13}).
In the following, $[n]:=n \!\!\mod p$.

\begin{prop}\label{prop:3.2}
Every $f\in\mc{F}_{\tau,p}$ is of the form
\begin{equation}\label{eq:series}
f(z)=\sum_{n\in\Z}e^{2\pi inx}
e^{\pi in^2\frac{\omega_x}{p}}f_{[n]}(y+n\tfrac{\omega_y}{p}) \;,
\end{equation}
for some Schwartz functions $f_{[1]},\ldots,f_{[p]}\in\mathcal{S}(\R)$,
uniquely determined by $f$ and given by:
\begin{equation}\label{eq:inver}
f_{[n]}(y)=
\int_0^1 e^{-2\pi inx}e^{\pi in^2\frac{\omega_x}{p}}f(z-\tfrac{n}{p}\tau)\de x
\qquad\forall\;n=1,\ldots,p.
\end{equation}
We denote by $\varphi_{\tau,p}:\mc{F}_{\tau,p}\to\HH_{\tau,p}:=\mc{S}(\R)\otimes\C^p$
the bijection associating to $f$ the vector valued function $\u{f}=(f_{[1]},\ldots,f_{[p]})^t$.
\end{prop}

\begin{proof}
Since $f$ is a smooth periodic function of $x$, then
$
f(z)=\sum_{n\in\Z}e^{2\pi inx}\widetilde f_n(y) ,
$
for some functions $\widetilde f_n$ of $y$. The condition
$$
f(z+\tau)=\sum_{n\in\Z}e^{2\pi in(x+\omega_x)}\widetilde f_n(y+\omega_y)
=e^{-\pi ip(\omega_x+2x)}f(z)
=\sum_{n\in\Z}e^{-\pi ip\omega_x}e^{2\pi i(n-p)x}\widetilde f_n(y) \vspace{-5pt}
$$
gives the recursive relation
$
\widetilde f_{n+p}(y)=e^{\pi i (2n+p)\omega_x}\widetilde f_n(y+\omega_y) \;.
$
If we write $n=j+kp$ with $k\in\Z$ and $0\leq j<p$, then
previous equation has solution
$$
\widetilde f_{j+kp}(y)=e^{\pi i(2jk+k^2p)\omega_x}\widetilde f_j(y+k\omega_y) \;.
$$
Called $f_{[j]}(y)=e^{-\pi i\frac{j^2}{p}\omega_x}\widetilde f_j(y-\frac{j}{p}\omega_y)$,
for $j=0,\ldots,p-1$,
after the replacement $j+kp\to n$ we get \eqref{eq:series}. Using Fourier inversion formula
we get \eqref{eq:inver}. We must now prove that $f$ is smooth if and only if $f_{[1]},\ldots,f_{[p]}$
are Schwartz functions.

To simplify the notations, let us give the proof for $\tau=i$ and $p=1$, but the generalization to arbitrary $p$ and $\tau$
is straightforward. So, what we want to prove is that:

\begin{lemma}
$\sum_{n\in\Z}e^{2\pi inx}\widetilde{f}(y+n)$ converges to a function
$f\in C^\infty(\R^2)\iff\widetilde{f}\in\mc{S}(\R)$.
\end{lemma}

\noindent
Let $g_n(z)=e^{2\pi inx}\widetilde{f}(y+n)$.
By Weierstrass M-test, $\sum_ng_n(z)$ converges (uniformly) to a $C^\infty$-function
if for all $j,k$ there exists a sequence of non-negative numbers $C^{j,k}_n$
which is summable in $n$ and such that:
$$
|\partial_x^j\partial_y^kg_n(z)|\leq C^{j,k}_n \;.
$$
Fix an $y_0>0$. For all $|y|\leq y_0$ we have $|n|=|y+n-y|\leq |y+n|+|y|\leq |y+n|+y_0$,
and:
$$
|\partial_x^j\partial_y^kg_n(z)|=
(2\pi |n|)^j|\partial_y^k\widetilde{f}(y+n)|\leq
(2\pi)^jn^{-2}\sum\nolimits_{l=0}^{j+2}\tbinom{j+2}{l}y_0^{j+2-l}
|(y+n)^l\partial_y^k\widetilde{f}(y+n)| \;.
$$
for all $n\neq 0$. If $\widetilde{f}$ is Schwartz, 
$\|\widetilde{f}\|_{l,k}:=\sup_y|y^l\partial_y^k\widetilde{f}(y)|$ is finite for all $l,k$, and
$$
|\partial_x^j\partial_y^kg_n(z)|\leq
(2\pi)^jn^{-2}\sum\nolimits_{l=0}^{j+2}\tbinom{j+2}{l}y_0^{j+2-l}\|\widetilde{f}\|_{l,k}
=K^{j,k}_{y_0}n^{-2} \;.
$$
If we set $C^{j,k}_n=K^{j,k}_{y_0}n^{-2}$ for $n\neq 0$, this series is summable.
Thus $\sum_ng_n(z)$ is uniformly convergent (together with its derivatives)
to a $C^\infty$-function on any strip $\R\times[-y_0,y_0]\subset\R^2$, and then
on the whole $\R^2$. This proves the ``$\Leftarrow$'' part.

\smallskip

Assume now the convergence of the series to a $C^\infty$-function $f$. By standard
Fourier analysis, $\widetilde{f}(y+n)=\int_0^1e^{-2\pi inx}f(z)\de x$ and, integrating
by parts,
$$
(y+n)^j\partial_y^k\widetilde{f}(y+n)=\Big(\frac{-i\,}{\,2\pi}\Big)^j\int_0^1e^{-2\pi inx}\nabla_1^j\nabla_2^kf(z)\de x \;.
$$
Recall that $\nabla_1^j\nabla_2^kf\in\mc{F}_{\tau,p}$ is a bounded function.
Let $C_{j,k}$ be its sup norm. Then $\|\widetilde{f}\|_{j,k}\leq (2\pi)^{-j}C_{j,k}$, proving
that $\widetilde{f}$ is a Schwartz function.
\end{proof}

$\mc{F}_{\tau,p}$ is a right $C^\infty(E_\tau)$-module. Given two elements
$f,g$, it follows from \eqref{eq:expl} that the product $f^*g$ is
a periodic function, hence an element of $C^\infty(E_\tau)$. The Hermitian
structure $(f,g)\mapsto f^*g$ turns $\mc{F}_{\tau,p}$ into a
pre-Hilbert module.

Note that $C^\infty(E_\tau)$ is generated by the two unitaries
\begin{equation}\label{eq:uv}
u_\tau(x,y):=e^{2\pi i(x-\frac{\omega_x}{\omega_y}y)} \;,\qquad\quad
v_\tau(x,y):=e^{2\pi i\frac{1}{\omega_y}y} \;,
\end{equation}
that reduces to the basic ones generating $C(\T^2)$ when $\tau=i$. This because
the diffeomorphism $$z\mapsto z'=(x-\tfrac{\omega_x}{\omega_y}y)+i\tfrac{1}{\omega_y}y$$
transforms $\Lambda$ into $\Z+i\Z$, and $E_\tau$ into the torus $\T^2$.

\begin{prop}\label{prop:hat}
$\HH_{\tau,p}$ is a right pre-Hilbert $C^\infty(E_\tau)$-module, isomorphic to $\mc{F}_{\tau,p}$
via the map $\varphi_{\tau,p}$, if we define the module structure (using Weyl operators) as
\begin{equation}\label{eq:modstru}
\u{f}\za u_\tau:=\big\{W(\tfrac{\omega_y}{p},-\tfrac{\omega_x}{\omega_y})\otimes S\big\}\u{f}
\;,\qquad
\u{f}\za v_\tau:=\big\{W(0,\tfrac{1}{\omega_y})\otimes C^*\big\}\u{f} \;,
\end{equation}
where $C,S\in M_p(\C)$ are the clock and shift operators in \eqref{eq:CS}.

The pullback with $\varphi_{\tau,p}$ of the canonical Hermitian structure of $\mc{F}_{\tau,p}$
is the Hermitian structure on $\HH_{\tau,p}$ given by
\begin{equation}\label{eq:herstr}
\inner{\u{f},\u{g}}=
\frac{1}{\omega_y}\sum_{m,n\in\Z}
u_\tau^mv_\tau^n
\int_{-\infty}^\infty
(\u{f}|\u{g}\za u_\tau^{-m}v_\tau^{-n})_t\,
\de t \;.
\end{equation}
where $(\,|\,):\HH_{\tau,p}\times\HH_{\tau,p}\to \mathcal{S}(\R)$ is
the canonical Hermitian structure of $\HH_{\tau,p}=\mathcal{S}(\R)\otimes\C^p$
as a right pre-Hilbert $\mathcal{S}(\R)$-module, given by \eqref{eq:canher}.
\end{prop}

\begin{proof}
From \eqref{eq:series} and \eqref{eq:uv} we get
\begin{align*}
v_\tau(z)f(z) &=e^{2\pi i\frac{1}{\omega_y}y}\sum_{n\in\Z}e^{2\pi inx}
e^{\pi in^2\frac{\omega_x}{p}} f_{[n]}(y+n\tfrac{\omega_y}{p}) 
\\
&=\sum_{n\in\Z}e^{2\pi inx}
e^{\pi in^2\frac{\omega_x}{p}}
e^{-2\pi in/p}
e^{2\pi i\frac{1}{\omega_y}(y+n\frac{\omega_y}{p})} f_{[n]}(y+n\tfrac{\omega_y}{p})
\\
&=\sum_{n\in\Z}e^{2\pi inx}
e^{\pi in^2\frac{\omega_x}{p}} f'_{[n]}(y+n\tfrac{\omega_y}{p}) 
\;,
\end{align*}
where
$$
 f'_{[n]}(y)=
e^{-2\pi in/p}e^{2\pi i\frac{1}{\omega_y}y} f_{[n]}(y) \;.
$$
If we call $\u{f}=\varphi_{\tau,p}(f)=( f_{[1]}, f_{[2]},\ldots, f_{[p]})^t$,
then one can check that $\u{f}'=\varphi_{\tau,p}(v_\tau f)=( f'_{[1]}, f'_{[2]},\ldots, f'_{[p]})^t
=\big\{W(0,\tfrac{1}{\omega_y})\otimes C^*\big\}\u{f}$, that is the second formula
in \eqref{eq:modstru}. Similarly,
\begin{align*}
u_\tau(z)f(z) &=e^{2\pi i(x-\frac{\omega_x}{\omega_y}y)}\sum_{n\in\Z}e^{2\pi inx}
e^{\pi in^2\frac{\omega_x}{p}} f_{[n]}(y+n\tfrac{\omega_y}{p}) \\
&=\sum_{n\in\Z}e^{2\pi i(n+1)x}
e^{\pi i(n+1)^2\frac{\omega_x}{p}}
e^{-\pi i\frac{\omega_x}{p}}
e^{-2\pi i\frac{\omega_x}{\omega_y}(y+n\frac{\omega_y}{p})}
 f_{[n]}(y+n\tfrac{\omega_y}{p}) \\
&=\sum_{n'\in\Z}e^{2\pi in'x}
e^{\pi in'^2\frac{\omega_x}{p}}
e^{\pi i\frac{\omega_x}{p}}
e^{-2\pi i\frac{\omega_x}{\omega_y}(y+n'\tfrac{\omega_y}{p})}
 f_{[n'-1]}(y-\tfrac{\omega_y}{p}+n'\tfrac{\omega_y}{p})
\;,
\end{align*}
where $n'=n+1$. If we call
$$
 f''_{[n]}(y)=
e^{\pi i\frac{\omega_x}{p}}
e^{-2\pi i\frac{\omega_x}{\omega_y}y}
 f_{[n-1]}(y-\tfrac{\omega_y}{p})
\;,
$$
then
$$
u_\tau(z)f(z)=\sum_{n\in\Z}e^{2\pi inx}
e^{\pi in^2\frac{\omega_x}{p}} f''_{[n]}(y+n\tfrac{\omega_y}{p}) \;.
$$
One can check that 
$$
 f''_{[n]}=W(\tfrac{\omega_y}{p},-\tfrac{\omega_x}{\omega_y}) f_{[n-1]} \;,
$$
proving the first equation in \eqref{eq:modstru}.

It remains compute $f^*g$. Using \eqref{eq:series} again we get:
\begin{align*}
(f^*g)(z) &=\sum_{m,n\in\Z}e^{2\pi i(n-m)x}
e^{\pi i(n^2-m^2)\frac{\omega_x}{p}}
 f_{[m]}^*(y+m\tfrac{\omega_y}{p})
 g_{[n]}(y+n\tfrac{\omega_y}{p}) \\
&=\sum_{k,m\in\Z}e^{2\pi ikx}
e^{\pi i(k^2+2km)\frac{\omega_x}{p}}
 f_{[m]}^*(y+m\tfrac{\omega_y}{p})
 g_{[m+k]}(y+(m+k)\tfrac{\omega_y}{p}) \;,
\end{align*}
where we called $n=m+k$.
This can be rewritten as follows. Using the identity
\begin{equation}\label{eq:Fy}
F(y)=\frac{1}{\omega_y}\sum_{n\in\Z}e^{2\pi in\frac{1}{\omega_y}y}\int_0^{\omega_y}
e^{-2\pi in\frac{1}{\omega_y}t} F(t)\de t
\end{equation}
we get
$$
f^*g=\frac{1}{\omega_y}\sum_{k,m,n\in\Z}u_\tau^k v_\tau^n
e^{\pi i(k^2+2km)\frac{\omega_x}{p}}
\int_0^{\omega_y}
e^{2\pi i\frac{1}{\omega_y}(k\omega_x-n)t}
 f_{[m]}^*(t+m\tfrac{\omega_y}{p})
 g_{[m+k]}(t+(m+k)\tfrac{\omega_y}{p})\de t \;.
$$
Let $\u{f}=( f_{[0]}, f_{[1]},\ldots, f_{[p-1]})^t$ as above,
and $\u{g}=( g_{[0]}, g_{[1]},\ldots, g_{[p-1]})^t$.
Note that
\begin{align*}
\u{g}\za u_\tau^{-k}v_\tau^{-n} &=
\big\{W(\tfrac{\omega_y}{p},-\tfrac{\omega_x}{\omega_y})\otimes S\big\}^{-k}
\big\{W(0,\tfrac{1}{\omega_y})\otimes C^*\big\}^{-n}\u{g} \\
&=\big\{W(-\tfrac{k\omega_y}{p},\tfrac{k\omega_x}{\omega_y})
W(0,-\tfrac{n}{\omega_y})\otimes S^{-k}C^n\big\}\u{g} \\
&=e^{-\pi i kn/p}\big\{W(-\tfrac{k\omega_y}{p},\tfrac{k\omega_x-n}{\omega_y})\otimes S^{-k}C^n\big\}\u{g} \;,
\end{align*}
where we used \eqref{eq:Weylprod}. The $[m]$-th component of
$\u{g}\za u_\tau^{-k}v_\tau^{-n}$ evaluated in $t+m\tfrac{\omega_y}{p}$ is
\begin{multline*}
e^{\pi i\frac{n}{p}(2m+k)}
\big\{W(-\tfrac{k\omega_y}{p},\tfrac{k\omega_x}{\omega_y}
-\tfrac{n}{\omega_y}) g_{[m+k]}\big\}(t+m\tfrac{\omega_y}{p})
\\[5pt]
=
e^{\pi i(k^2+2km)\frac{\omega_x}{p}}
e^{2\pi i\frac{1}{\omega_y}(k\omega_x-n)t}
 g_{[m+k]}(t+(m+k)\tfrac{\omega_y}{p}) \;.
\end{multline*}
Hence
$$
f^*g=
\frac{1}{\omega_y}\sum_{k,m,n\in\Z}u_\tau^kv_\tau^n\int_0^{\omega_y}
\{\u{g}\za u_\tau^{-k}v_\tau^{-n}\cdot\u{f}^*\}_{[m]}(t+m\tfrac{\omega_y}{p})\de t \;.
$$
Let $m=r+sp$ with $r=0,\ldots,p-1$ and $s\in\Z$. Using
$$
\sum_{s\in\Z}\int_0^{\omega_y}F(t+s\omega_y)\de t=\int_{-\infty}^\infty F(t)\de t
$$
we get
$$
f^*g=
\frac{1}{\omega_y}\sum_{k,n\in\Z}u_\tau^kv_\tau^n\int_{-\infty}^{\infty}
\sum_{r=0}^{p-1}\{\u{g}\za u_\tau^{-k}v_\tau^{-n}\cdot\u{f}^*\}_{[r]}(t+r\tfrac{\omega_y}{p})\de t \;.
$$
By translation invariance of the measure, we can replace $t+r\tfrac{\omega_y}{p}$
by $t$ and get $f^*g=\inner{\u{f},\u{g}}$, where $\inner{\u{f},\u{g}}$ is given by \eqref{eq:herstr}.
\end{proof}


\subsection{Some remark on holomorphic quasi-periodic functions}\label{sec:4.3}
If $\alpha$ is any factor of automorphy, there is a corresponding line bundle $L_\alpha\to E_\tau$ with
total space $L_\alpha=\C\times\C/\!\sim$, where the equivalence relation is
$$
(z+\lambda,w)\sim (z,\alpha(\lambda,z)w) \;, \qquad \forall\;z,w\in\C,\lambda\in\Lambda,
$$
and with projection sending the class of $(z,w)$ to the class of $z$ in $E_\tau$.
The set of smooth sections of this line bundle can be identified with the set \eqref{eq:smsec}.
It is easy to characterize which $L_\alpha$ have (non-zero) holomorphic sections.

\begin{prop}
If $\,\Gamma_\alpha\!$ contains non-zero holomorphic functions, then there exists an open
set where all the functions $z\mapsto\alpha(\lambda,z)$ are holomorphic (for all $\lambda\in\Lambda$).
\end{prop}

\begin{proof}
Let $\bar\partial=\frac{1}{2}(\partial_x+i\partial_y)$ and suppose $f\in\Gamma_\alpha$ is holomorphic and not identically zero.
By deriving the relation $f(z+\lambda)=\alpha(\lambda,z)f(z)$ one gets
$$
(\bar\partial\alpha)f=0 \;,
$$
hence there exists an open set (the support of $f$), independent of $\lambda$,
where $\bar\partial\alpha(\lambda,z)=0$ for all $\lambda\in\Lambda$.
\end{proof}

As a corollary, the unitary factor of automorphy $\beta^p$, with $\beta$ as in \eqref{eq:beta},
gives a module $\mathcal{F}_{\tau,p}=\Gamma_{\beta^p}$ 
with no non-zero holomorphic elements, for any $p\in\Z\smallsetminus\{0\}$.
The $C^\infty(E_\tau)$-module map $f\mapsto e^{\pi py^2/\omega_y}f$ sends $\mathcal{F}_{\tau,p}=\Gamma_{\beta^p}$ to the module
$\Gamma_{\alpha^p}$, with $\alpha$ as in \eqref{eq:alpha}. It is well-known, and easy to check explicitly
(see Prop.~\ref{prop:holel} below), that $\Gamma_{\alpha^p}$ has non-zero holomorphic elements, proving that the
line bundles $L_{\alpha^p}$ and $L_{\beta^p}$ are isomorphic as smooth line bundles, but not as holomorphic line bundles.
Similarly to \eqref{eq:series}, we have:

\begin{prop}\label{prop:holel}
Every $f\in\Gamma_{\alpha^p}$ is of the form
\begin{equation}\label{eq:seriesHol}
f(z)=e^{\pi py^2/\omega_y}\sum_{n\in\Z}e^{2\pi inx}
e^{\pi in^2\frac{\omega_x}{p}}f_{[n]}(y+n\tfrac{\omega_y}{p}) \;,
\end{equation}
for some Schwartz functions $f_{[1]},\ldots,f_{[p]}\in\mathcal{S}(\R)$.
\end{prop}

As a consequence,

\begin{prop}
Let $q=e^{\pi i\tau}$ and $H^0(\Gamma_{\alpha^p},\bar\partial)$ the set of holomorphic
functions in $\Gamma_{\alpha^p}$.
If $p>0$, every $f\in H^0(\Gamma_{\alpha^p},\bar\partial)$ is of the form
$$
f(z)=\sum_{n\in\Z}c_{[n]}q^{n^2/p}e^{2\pi inz} \;,
$$
where $c_{[n]}$ are complex numbers. 
The dimension of $H^0(\Gamma_{\alpha^p},\bar\partial)$ is $p$ if $p>0$, and $0$ if $p<0$.
\end{prop}

\begin{proof}
From \eqref{eq:seriesHol} we get
$$
2\bar\partial f(z)=ie^{\pi ipy^2/\omega_y}\sum_{n\in\Z}e^{2\pi inx}
e^{\pi in^2\frac{\omega_x}{p}}
(\tfrac{2\pi p}{\omega_y} y+2\pi n+\partial_y)
f_{[n]}(y+n\tfrac{\omega_y}{p}) \;.
$$
Thus $\bar\partial f=0$ if and only if
$$
(\tfrac{2\pi p}{\omega_y} y+\partial_y)f_{[n]}(y)=0
$$
for all $y\in\R$. The general solution is $f_{[n]}(y)=c_{[n]}e^{-\pi py^2/\omega_y}$, 
which is Schwartz if and only if $p>0$,
hence the thesis.
\end{proof}


\subsection{Connections and local trivialization}

A connection $\nabla$ on $\mathcal{F}_{\tau,p}$ 
is given by $\nabla f=(\nabla_1f)\de x+(\nabla_2f)\de y$, where
$\nabla_1$ and $\nabla_2$ are given in \eqref{eq:nabla}.
One can explicitly check that the Leibniz rule is satisfied (hence, the above
formulas define indeed a connection).
The corresponding connection $1$-form is $\omega=\frac{2\pi ip}{\omega_y}y\hspace{1pt}\de x$
(living on the covering $\C$ of $E_\tau$) and the curvature is
$$
\Omega=\de\omega=-\frac{2\pi ip}{\omega_y}\de x\wedge\de y\;.
$$
Integrating over a fundamental domain we get
$-\tfrac{2\pi ip}{\omega_y}$ times the area of the parallelogram
with vertices $0$, $1$, $\tau$ and $1+\tau$ (that is equal to $\omega_y$). Thus
$$
\int_{E_\tau}\frac{i}{2\pi}\,\Omega=p \;,
$$
proving that the Chern number of the corresponding line bundle is $p$.

\smallskip

It is an interesting exercise to do a doublecheck using a local trivialization, since it
gives us as a byproduct an explicit formula for a projection representing the $K$-theory
class of the module $\mathcal{F}_{\tau,p}$. 
From Serre-Swan theorem, if $g_{jk}:U_j\cap U_k\to \C^*$ are transition
functions of the line bundle relative to a (finite) open cover $\{U_j\}_{j=1}^n$ of $E_\tau$,
and $\psi_j\in C(E_\tau)$ are such that $\{|\psi_j|^2\}_{j=1}^n$ is a partition of unity subordinated
to the cover, then an idempotent matrix is given by (no summation implied):
$$
P=(\psi_jg_{jk}\overline{\psi}_k) \;.
$$
The matrix elements of $P$ are global continuous functions on $E_\tau$, even if $g_{jk}$ in
general are not. The cocycle condition for the transition functions guarantees that $P^2=P$,
and one can prove that the module of sections of the line bundle is isomorphic to the finitely
generated projective module associated to $P$~\cite{Swan62}.
Choosing smooth transition functions and partition of unity, one gets a
smooth idempotent. If $g_{jk}$ have values in $U(1)$ rather than $\C^*$,
the idempotent is a projection.

\smallskip

Since here the complex structure is irrelevant, to simplify the discussion let us set $\tau=i$.
The fundamental domain is then a square,
and we can identify $\T^2$ with \mbox{$[0,1]\times [0,1]/\!\!\sim$}, where the equivalence
relation is $(x,0)\sim (x,1)$ and $(0,y)\sim (1,y)$ for all $x,y$. 
The line bundle with factor of automorphy $\beta^p$ can be described as follows.
The total space is \mbox{$[0,1]\times [0,1]\times\C/\!\!\sim$}, where
$(0,y,w)\sim (1,y,w)$ and $(x,0,w)\sim (x,1,e^{-2\pi ipx}w)$
for all $x,y,w$.
We choose two charts $U_1=\{0<y<1\}$ and $U_2=\{y\neq\frac{1}{2}\}$ on $\T^2$.
Two local sections $s_1$ and $s_2$ are defined as follows:
\begin{align*}
s_1 &: U_1\to\C \;,\qquad s_1(x,y)=1 \\
s_2 &: U_2\to\C \;,\qquad s_2(x,y)=\begin{cases}
1 &\mathrm{if}\;0\leq y<\frac{1}{2},\\
e^{2\pi ipx} &\mathrm{if}\;\frac{1}{2}<y\leq 1.
\end{cases}
\end{align*}
Since $s_i(0,y)=s_i(1,y)$, for $i=1,2$, and $s_2(x,0)=e^{-2\pi ipx}s_2(x,1)$,
the sections are well defined. From these, we can compute the transition function of the bundle.
On $U_1\cap U_2$, $s_1(x,y)=g(x,y)s_2(x,y)$ where the transition function is
$$
g(x,y)=\begin{cases}
1 &\mathrm{if}\;0<y<\frac{1}{2},\\
e^{2\pi ipx} &\mathrm{if}\;\frac{1}{2}<y<1.
\end{cases}
$$
For any smooth choice of $\{\psi_1,\psi_2\}$, if we call $\Psi=(\psi_1,g\psi_2)^t$,
the corresponding projection is $P=\Psi\Psi^\dag$, with row-by-column multiplication understood.
The Grasmannian connection has connection $1$-form
\begin{multline*}
\Psi^\dag\de \Psi=
\overline\psi_1\de\psi_1+
\overline\psi_2\de\psi_2+
|\psi_2|^2 g^*\de g \\[3pt]
=\tfrac{1}{2}\de(|\psi_1|^2+|\psi_2|^2)+|\psi_2|^2g^*\de g=
\tfrac{1}{2}\de 1+|\psi_2|^2g^*\de g=|\psi_2|^2g^*\de g \;.
\end{multline*}
An explicit computation gives
$$
\Psi^\dag\de \Psi=
\begin{cases}
0 &\mathrm{if}\;0<y\leq \frac{1}{2},\\
2\pi i |\psi_2|^2 p\hspace{1pt}\de x &\mathrm{if}\;\frac{1}{2}\leq y<1.
\end{cases}
$$
Note that this is only well defined on the chart $U_1$
(it's zero for $y=\frac{1}{2}$).
The curvature is
$$
\Omega=
\begin{cases}
0 &\mathrm{if}\;0<y\leq\frac{1}{2},\\
-2\pi ip\,\de x\wedge\de  |\psi_2|^2 &\mathrm{if}\;\frac{1}{2}\leq y<1.
\end{cases}
$$
Since $\psi_j$ vanishes outside $U_j$ and $|\psi_1(y)|^2+|\psi_2(y)|^2=1$,
then $\psi_2(\frac{1}{2})=0$ and from 
$\psi_1(1)=0$ we get $\psi_2(1)^2=1$. Then:
$$
\int_{\T^2}\frac{i}{2\pi}\,\Omega=p
\int_0^1\de x\int_{1/2}^1\de |\psi_2|^2=
p\,|\psi_2(1)|^2-p\,|\psi_2(\tfrac{1}{2})|^2=p \;,
$$
confirming that the degree is $p$.


\section{Vector bundles on the noncommutative torus}\label{sec:5}
We now come back to the pre $C^*$-algebra $A^\infty_\theta$. 
As one can easily check, $C^\infty(\R^2)$ is a $A^\infty_\theta$-module,
with bimodule structure given on generators by:
\begin{equation}\label{eq:bimodB}
\begin{split}
(U\az f)(x,y) =e^{2\pi ix}f\big(x,y+\tfrac{1}{2}\theta\big) \;,&\qquad\quad
(f\za U)(x,y) =e^{2\pi ix}f\big(x,y-\tfrac{1}{2}\theta\big) \;,
\\[2pt]
(V\az f)(x,y) =e^{2\pi iy} f\big(x-\tfrac{1}{2}\theta,y\big) \;,&\qquad\quad
(f\za V)(x,y) =e^{2\pi iy}f\big(x+\tfrac{1}{2}\theta,y\big) \;.
\end{split}\hspace*{-10pt}
\end{equation}
Let $J$ be the antilinear involutive map:
$$
(Jf)(x,y)=\overline{f(-x,-y)} \;.
$$
As one can check on generators, conjugation by $J$ sends the algebra $A^\infty_\theta$
into its commutant, and in particular transforms the left action into the right action
and viceversa.

\begin{rem}
The relation between \eqref{eq:bimodB} and Moyal product \eqref{eq:Moyalprod} is the following.
The space $\B(\R^2)$ of bounded smooth functions with all derivatives bounded is an $A^\infty_\theta$
sub-bimodule of $C^\infty(\R^2)$.
For $a\in A^\infty_\theta$ and $f\in\B(\R^2)$, one can check that
$$
a\az f=T_\theta^{-1}(a)\ast_\theta f
\qquad
\text{and}
\qquad
f\za a=f\ast_\theta T_\theta^{-1}(a) \;,
$$
where $\ast_\theta$ is the product \eqref{eq:Moyalprod} and $T_\theta$ the
quantization map \eqref{eq:Ttheta}. In this case, the fact that \eqref{eq:bimodB} define
a bimodule is a consequence of associativity of Moyal product.
\end{rem}

We will focus on right modules, but similar results hold for left modules as well.
Although as one might expect the subspaces $\mc{F}_{i,p}\subset C^\infty(\R^2)$
are not submodules for the right module structure \eqref{eq:bimodB}, we can look for submodules
of the form $\mc{F}_{\tau,p}$ for some non-trivial value of the modular parameter $\tau$.

\begin{prop}\label{prop:mainres}
The vector space $\mc{F}_{\tau,p}$ is a right $A_\theta^\infty$-module if and only if
$$
\tau-\tfrac{p\theta}{2}i \in\Z+i\Z \;.
$$
\end{prop}

\begin{proof}
We are looking for necessary and sufficient conditions on $\tau$ such that, for any
$f$ satisfying \eqref{eq:expl}, $f\za U$ and $f\za V$ satisfy \eqref{eq:expl} too.
If $f$ satisfies \eqref{eq:expl}, from \eqref{eq:bimodB} we get:
\begin{align*}
(f\za U)(z+1)    &=(f\za U)(z) \;,&
(f\za U)(z+\tau) &=e^{2\pi i\omega_x} e^{-\pi ip(\omega_x+2x)}(f\za U)(z) \;,\\[2pt]
(f\za V)(z+1)    &=(f\za V)(z) \;,&
(f\za V)(z+\tau) &=e^{2\pi i(\omega_y-\frac{p\theta}{2})}e^{-\pi ip(\omega_x+2x)}(f\za V)(z) \;,
\end{align*}
where as usual $z=x+iy$ and $\tau=\omega_x+i\omega_y$, and we used the property
\eqref{eq:expl} of $f$. So, the first condition in \eqref{eq:expl} is always
satisfied by $f\za U$ and $f\za V$, while the second is satisfied if and only if
$\omega_x$ and $\omega_y-\frac{p\theta}{2}$ are integers, that is what we wanted
to prove.
\end{proof}

Now that we established that the vector spaces in Prop.~\ref{prop:mainres} are
$A_\theta^\infty$-modules, we want to give a description in terms of Weyl
operators, in order to compare them with Def.~\ref{def:heismod}. In the
rest of this section, we assume that
$$
\tau=r+i(s+\tfrac{p\theta}{2})
$$
for some $r,s\in\Z$. Using the vector space isomorphism $\varphi_{\tau,p}:\mc{F}_{\tau,p}\to\HH_{\tau,p}:=\mc{S}(\R)\otimes\C^p$
of Prop.~\ref{prop:3.2} we transport the right module structure \eqref{eq:bimodB} to $\HH_{\tau,p}$.
Let
$$
\u{f}\bza a=\varphi_{\tau,p}\big(\varphi_{\tau,p}^{-1}(\u{f})\za a\big)
$$
for all $a\in A^\infty_\theta$
and $\u{f}\in\HH_{\tau,p}$.

\begin{prop}
For all $\u{f}\in\HH_{\tau,p}$,
\begin{equation}\label{eq:modstruTheta}
\u{f}\bza U=e^{\pi i\frac{r}{p}}\big\{W(\tfrac{s}{p}+\theta,0)\otimes (C^*)^rS\big\}\u{f}
\;,\qquad
\u{f}\bza V=\big\{W(0,1)\otimes (C^*)^s\big\}\u{f} \;,
\end{equation}
where $C,S\in M_p(\C)$ are the clock and shift operators.
\end{prop}

\begin{proof}
Here $\omega_x=r$ and $\omega_y=s+\frac{p\theta}{2}$, with $r,s\in\Z$.
From \eqref{eq:series} and \eqref{eq:bimodB}, for all $f\in\mc{F}_{\tau,p}$:
\begin{align*}
(f\za V)(x,y) &=
e^{2\pi iy}\sum_{n\in\Z}e^{2\pi inx}e^{\pi in\theta}
e^{\pi in^2\frac{\omega_x}{p}} f_{[n]}(y+n\tfrac{\omega_y}{p}) \\
&=
\sum_{n\in\Z}e^{2\pi inx}e^{\pi in^2\frac{\omega_x}{p}}
\bigl\{e^{-2\pi ins/p}W(0,1) f_{[n]}\bigr\}(y+n\tfrac{\omega_y}{p}) \;.
\end{align*}
This proves the second equation in \eqref{eq:modstruTheta}.
Concerning the first one, using \eqref{eq:series} and \eqref{eq:bimodB}:
\begin{align*}
(f\za U)(x,y) &=e^{2\pi ix}\sum_{n\in\Z}e^{2\pi inx}
e^{\pi in^2\frac{\omega_x}{p}} f_{[n]}(y+n\tfrac{\omega_y}{p}-\tfrac{1}{2}\theta) \\
\intertext{and with the replacement $n\to n-1$:}
&=\sum_{n\in\Z}e^{2\pi inx}
e^{\pi i(n^2-2n+1)\frac{\omega_x}{p}} f_{[n-1]}(y+n\tfrac{\omega_y}{p}-\tfrac{s}{p}-\theta) \\
&=\sum_{n\in\Z}e^{2\pi inx}
e^{\pi i(n^2-2n+1)\frac{\omega_x}{p}}
\big\{W(\tfrac{s}{p}+\theta,0) f_{[n-1]}\big\}(y+n\tfrac{\omega_y}{p}) \;.
\end{align*}
This proves the first equation in \eqref{eq:modstruTheta}.
\end{proof}

For $\theta=0$ and $\tau=i$ (that means $r=0$, $s=1$), we recover the module of Prop.~\ref{prop:hat}.
For arbitrary $\theta$, if $r=0$ and $s\neq 0$ then the map $T:\C^p\to\C^p$ given by
$$
(Tw)_n=w_{-sn\;\mathrm{mod}\;p}
$$
is invertible (in fact unitary), and $\mathsf{id}_{\mathcal{S}(\R)}\otimes T$
is a right $A^\infty_\theta$-module isomorphism between $\HH_{\tau,p}$
with module structure \eqref{eq:modstruTheta} and the module $\E_{p,s}$
of Def.~\ref{def:heismod}.
If $r=0$ and $s=0$ then the module $\HH_{\tau,1}$ coincides with $\E_{1,0}$.

Notice that, modulo an isomorphism, every Heisenberg module of \S\ref{sec:4.1} is the deformation
of a line bundle (no vector bundles of higher rank appear).

To conclude this section, we now generalize the second part of Prop.~\ref{prop:hat}
and show how to get a Hermitian structure similar to \eqref{eq:hstr} from a canonical (and
simpler) Hermitian structure on $\mc{F}_{\tau,p}$. In the classical case, for all
$f,g\in\mc{F}_{\tau,p}$ the product $\overline{f}g$ is periodic, and the map $(f,g)\mapsto \overline{f}g$
is a Hermitian structure on $\mc{F}_{\tau,p}$ that transformed with the isomorphism $\varphi_{\tau,p}$
becomes \eqref{eq:herstr}. The noncommutative analog of this fact is the content of next proposition,
where the pointwise product $\overline{f}g$ is replaced by Moyal product $\overline{f}\ast_\theta g$.
This can be done for arbitrary Chern number $p$, but we must assume that $r=0$ and $s=1$,
meaning that we are focusing on modules isomorphic to ``the rank $1$'' Heisenberg
modules $\E_{p,1}$ (the reason is that, in the $y$ direction, $\overline{f}\ast_\theta g$ is periodic with period $s$,
so it belongs to $C^\infty(\T^2)$ only if $s=1$).

\begin{prop}
Let $\tau=i(1+\frac{1}{2}p\hspace{1pt}\theta)$. Then, for all $f,g\in\mc{F}_{\tau,p}$:
\begin{equation}\label{eq:canherS}
T_\theta(\overline{f}\ast_\theta g)=\sum\nolimits_{m,n\in\Z}V^nU^m\int_{-\infty}^{+\infty}
(\u{f}\bza V^nU^m|\u{g})_t\,\de t \;,
\end{equation}
where $(\,|\,)_t$ is canonical inner product in \eqref{eq:canher},
$\u{f}=\varphi_{\tau,p}(f)$ and $\u{g}=\varphi_{\tau,p}(g)$ as in Prop.~\ref{prop:3.2},
and $T_\theta:C^\infty(\T^2)\to A_\theta^\infty$ is the quantization map \eqref{eq:Ttheta}.
\end{prop}

\begin{proof}
For $\tau=i(s+\frac{1}{2}p\hspace{1pt}\theta)$, equations \eqref{eq:series} and \eqref{eq:Moyalprod} give:
\begin{multline*}
(\overline{f}\ast_\theta g)(z)=
\sum_{m,n\in\Z}\frac{4}{\theta^2}u^{n-m}\int
e^{-2\pi im\xi_1}
e^{2\pi in\eta_1}e^{\frac{4\pi i}{\theta}(\eta_1\xi_2-\eta_2\xi_1)}
\;\times\\ \times\;
f_{[m]}^*(y+\xi_2+m\tfrac{s}{p}+m\tfrac{\theta}{2})
g_{[n]}(y+\eta_2+n\tfrac{s}{p}+n\tfrac{\theta}{2})
\de\xi\de\eta
\\[5pt]
=\sum_{m,k\in\Z}u(x,y)^k
f_{[m]}^*(y+m\tfrac{s}{p}-k\tfrac{\theta}{2})
g_{[m+k]}(y+(m+k)\tfrac{s}{p}+k\tfrac{\theta}{2})
 \;,
\end{multline*}
where $\xi=\xi_1+i\xi_2$, $\eta=\eta_1+i\eta_2$ and we called $n=m+k$.
Using \eqref{eq:Fy} with $\omega_y=1$:
$$
\overline{f}\ast_\theta g
=\sum_{k,m,n\in\Z}u^kv^n\int_0^1e^{- 2\pi int}
f_{[m]}^*(t+m\tfrac{s}{p}-k\tfrac{\theta}{2})
g_{[m+k]}(t+(m+k)\tfrac{s}{p}+k\tfrac{\theta}{2})\de t \;.
$$
From \eqref{eq:modstruTheta}:
\begin{multline*}
\u{g}\bza U^{-k}V^{-n}=
\big\{W(0,1)\otimes (C^*)^s\big\}^{-n}
\big\{W(\tfrac{s}{p}+\theta,0)\otimes S\big\}^{-k}\u{g}
\\[3pt]
=\big\{W(0,-n)\otimes C^{ns}\big\}
\big\{W(-k\tfrac{s}{p}-k\theta,0)\otimes S^{-k}\big\}\u{g}
\\[3pt]
=e^{\pi ink(\frac{s}{p}+\theta)}\big\{W(-k\tfrac{s}{p}-k\theta,-n)\otimes C^{ns}S^{-k}\big\}\u{g}
 \;,
\end{multline*}
where we used \eqref{eq:WeylOP}. The $[m]$-th component, evaluated at $t+m\tfrac{s}{p}-k\tfrac{\theta}{2}$,
gives
\begin{multline*}
e^{\pi ink(\frac{s}{p}+\theta)}e^{2\pi im\frac{ns}{p}}
\big\{W(-k\tfrac{s}{p}-k\theta,-n)g_{[m+k]}\}(t+m\tfrac{s}{p}-k\tfrac{\theta}{2})
\\[5pt]
=e^{\pi ink\theta}e^{-2\pi int}
g_{[m+k]}\big(t+(m+k)\tfrac{s}{p}+k\tfrac{\theta}{2}\big)
 \;.
\end{multline*}
Hence for $s=1$:
\begin{multline*}
\overline{f}\ast_\theta g
=\sum_{k,m,n\in\Z}e^{-\pi ink\theta}u^kv^n\int_0^1
\big\{f_{[m]}^*(\u{g}\bza U^{-k}V^{-n})_{[m]}
\big\}(t+m\tfrac{s}{p}-k\tfrac{\theta}{2})\de t
\\[3pt]
=\sum_{k,n\in\Z}e^{-\pi ink\theta}u^kv^n\int_{-\infty}^{+\infty}
(\u{f}|\u{g}\bza U^{-k}V^{-n})_t\de t
 \;,
\end{multline*}
where $(\,|\,)_t$ is the Hermitian structure in \eqref{eq:canher}.
Using then \eqref{eq:Ttheta} one finds:
$$
T_\theta(\overline{f}\ast_\theta g)=\sum\nolimits_{m,n\in\Z}e^{-2\pi imn\theta}U^mV^n\int_{-\infty}^{+\infty}
(\u{f}|\u{g}\bza U^{-m}V^{-n})_t\de t \;.
$$
The observation that
\mbox{$\;(\u{f}|\u{g}\bza U^{-m}V^{-n})_t=(\u{f}\bza V^nU^m|\u{g})_t\;$}
and \mbox{$\;U^mV^n=e^{2\pi imn\theta}V^nU^m\;$} concludes the proof.
\end{proof}

Note that despite the similarity between \eqref{eq:canherS} and \eqref{eq:hstr},
in the two formulas two different actions of $A_\theta^\infty$ on $\mc{S}(\R)\otimes\C^p$
are used.

\subsection*{Acknowledgments}\vspace{-3pt}
We thank G.~Landi for discussions and correspondence.
This research was partially supported by UniNA and Compagnia di San Paolo under the grant ``STAR Program 2013''.



\vspace*{-3pt}

\begin{thebibliography}{9}
\small\itemsep=-1pt

\vspace*{-1pt}

\bibitem{Ati57}
M.F. Atiyah, \textit{Vector bundles over an elliptic curve}, Proc. London Math. Soc. 7 (1957), 414--452.

\bibitem{BL04}
C.~Birkenhake and H.~Lange, \textit{Complex Abelian Varieties}, 2nd ed., Springer, 2004.


\bibitem{Con80}
A.~Connes, \textit{$C^*$-alg\`ebres et g\'eom\'etrie diff\'erentielle}, C. R. Acad. Sci. Paris 290A (1980), 599--604.


\bibitem{Con94}
A.~Connes,
\textit{Noncommutative Geometry}, Academic Press, 1994.

\bibitem{CR87}
A.~Connes and M.A. Rieffel, \textit{Yang-Mills for noncommutative two-tori},
Contemp. Math. 62 (1987), 237--266.


\bibitem{Dri89}
V.G. Drinfeld, \textit{Quasi-Hopf algebras}, Leningrad Math. J. 1 (1989), 1419--1457.

\bibitem{Fio11}
G. Fiore, \textit{On twisted symmetries and quantum mechanics with
a magnetic field on noncommutative tori}, PoS(CNCFG2010)018,
\url{http://pos.sissa.it/archive/conferences/127/018/CNCFG2010_018.pdf}.

\bibitem{Fio13}
G. Fiore, \textit{On quantum mechanics with a magnetic field on $\mathbb{R}^n$ and on a torus $\mathbb{T}^n$, and their relation}, 	Int. J. Theor. Phys. 52 (2013), 877--896.

\bibitem{Fol89}
G.B.\ Folland, \textit{Harmonic analysis in phase space}, Ann.\ Math.\ Studies 122, Princeton Univ.\ Press, 1989.

\bibitem{GGISV04}
V.~Gayral, J.M.~Gracia-Bond{\'\i}a, B.~Iochum, T.~Sch{\"u}cker, J.C. V{\'a}rilly, \textit{Moyal planes are spectral triples}, Commun. Math. Phys. 246 (2004), 569--623.

\bibitem{LW11}
G. Lechner and S. Waldmann, \textit{Strict deformation quantization of locally convex algebras and modules}, preprint arXiv:1109.5950 [math.QA].

\bibitem{MvS08}
S. Mahanta and W.D. van Suijlekom, \textit{Noncommutative tori and the Riemann-Hilbert correspondence}, J. Noncommut. Geom. 3 (2009),  261--287.

\bibitem{Man04}
Y. Manin, \textit{Real multiplication and noncommutative geometry}, in ``The legacy of Niels Henrik Abel'', Springer Verlag, Berlin (2004), pp.~685--727.

\bibitem{Mum83}
D.~Mumford, \textit{Tata Lectures on Theta I}, Birk{\"a}user, 1983.

\bibitem{Pla06}
J. Plazas,
\textit{Arithmetic structures on noncommutative tori with real multiplication},
Int. Math. Res. Notices (2008), rnm147.

\bibitem{PS02}
A. Polishchuk and A. Schwarz,
\textit{Categories of holomorphic vector bundles on noncommutative two-tori},
Commun. Math. Phys. 236 (2003), 135--159.

\bibitem{Pol02}
A. Polishchuk,
\textit{Noncommutative two-tori with real multiplication as noncommutative projective varieties},
J. Geom. Phys. 50 (2004), 162--187.

\bibitem{Rie83}
M.A.~Rieffel, \textit{The cancellation theorem for projective modules
over irrational rotation $C^*$-algebras}, Proc. London Math. Soc. 47 (1983), 285--302.

\bibitem{Rie90}
M.A.~Rieffel,
\textit{Deformation quantization and operator algebras}, Proc. Symp. Pure Math. 51 (AMS, 1990), pp.~411--423.

\bibitem{Rie93}
M.A.~Rieffel,
\textit{Deformation quantization for actions of $\mathbb{R}^d$},
Mem. Amer. Math. Soc. 106, 
1993.

\bibitem{Swan62}
R.G.~Swan, \textit{Vector bundles and projective modules}, Trans.~Amer.~Math.~Soc.~\textbf{105} (1962) 264.


\bibitem{Var06}
J.C. V\'arilly, \textit{An introduction to noncommutative geometry}, EMS Lect. Ser. in Math., 2006.

\bibitem{Vla06}
M. Vlasenko,
\textit{The graded ring of quantum theta functions for noncommutative torus with real multiplication},
Int. Math. Res. Notices (2006), 15825.

\end{thebibliography}
\end{document}